\documentclass[12pt]{amsart}
\usepackage{amsmath}
\usepackage{amsfonts}
\usepackage{amssymb}
\usepackage{amsthm}
\usepackage[all]{xy}
\usepackage{color}

\newtheorem{theorem}{Theorem}[section]
\newtheorem{lemma}[theorem]{Lemma}
\newtheorem{proposition}[theorem]{Proposition}
\newtheorem{corollary}[theorem]{Corollary}
\newtheorem{definition}[theorem]{Definition}
\newtheorem{remark}[theorem]{Remark}

\numberwithin{equation}{section}
\begin{document}
\baselineskip=15.5pt

\title[Poincar\'e bundle]{Poincar\'e bundle for the fixed determinant moduli space on a nodal curve}

\author[Usha N. Bhosle]{Usha N. Bhosle\\
Indian Statistical Institute, Bangalore, India ,\\ 
Email: usnabh07@gmail.com}

%\address{Indian Statistical Institute, Bangalore, INDIA}

%\email{usnabh07@gmail.com; usha@math.tifr.res.in}



\begin{abstract}
            Let $Y$ be an integral nodal projective curve of arithmetic genus $g\ge 2$ with $m$ nodes defined over an algebraically closed field $k$ and $x$ a nonsingular closed point of $Y$.  Let $n$ and $d$ be coprime integers with $n\ge 2$. Fix a line bundle $L$ of degree $d$ on $Y$. Let $U_Y(n,d,L)$ denote the (compactified) "fixed determinant moduli space". We prove that the restriction $\mathcal{U}_{L,x}$ of the Poincare bundle  to $x \times U_Y(n,d,L)$ is stable with respect to the polarisation $\theta_L$ and its restriction to $x \times U'_Y(n,d,L)$, where $U'_Y(n,d,L)$ is the moduli space of vector bundles of rank $n$ and determinant $L$, is stable with respect to any polarisation. We show that the Poincar\'e bundle $\mathcal{U}_{L}$ on $Y \times U_Y(n,d,L)$ is stable with respect to the polarisation  $a \alpha + b \theta_L$  where $\alpha$ is a fixed ample Cartier divisor on $Y$ and $a, b$ are positive integers.  

\end{abstract}

\thanks{This work was done during my tenure as INSA Senior Scientist at Indian Statistical Institute, Bangalore.}

\maketitle

\section{introduction}

            Let $Y$ be an integral projective curve of arithmetic genus $g\ge 2$ defined over an algebraically closed field $k$ with at most nodes as singularities.  Let $n$ and $d$ be coprime integers with $n\ge 2$. Let $U_Y(n, d)$ denote the moduli space of stable torsion free sheaves of rank $n$ and degree $d$ on $Y$ and $U'_Y(n,d)$ its open dense sub variety corresponding to stable vector bundles.  For a fixed line bundle $L$ of degree $d$ on $Y$, let $U'_Y(n,d,L)$ denote the sub variety of $U'_Y(n,d)$ consisting of vector bundles with  determinant isomorphic to $L$.  Let $U_Y(n,d,L)$ be the closure of $U'_Y(n,d,L)$ in $U_Y(n,d)$ (with reduced structure).

    There exits a Poincar\'e sheaf 
         $$\mathcal{U}_L  \longrightarrow Y \times U_Y(n,d,L)\, ,$$ 
such that the restriction $\mathcal{U}_L\vert_{Y\times [E]} \cong E$ for any $E \in U_Y(n,d,L)$ \cite[Theorem 5.12']{N}.  Let $x$ denote a nonsingular  point of $Y$.  Then the restriction $\mathcal{U}_{L,x}$, of the Poincare sheaf to $x \times U_Y(n,d,L)$,  is a vector bundle on $x \times U_Y(n,d,L) \cong U_Y(n,d,L)$. There is a canonically defined ample line bundle $\theta$ on $U_Y(n,d)$ (the determinant of cohomology line bundle), let $\theta_L   \longrightarrow U_Y(n,d,L)$ be  its restriction. One has ${\rm Pic} \ U'_Y(n,d,L) \cong \mathbb{Z}$  \cite[Theorem I]{Bh1} and the restriction  of $\theta_L$ to $U'_Y(n,d,L)$ is the generator of Pic $U'_Y(n,d,L)$.  

          We do not know if ${\rm Pic} \ U_Y(n,d,L) \cong \mathbb{Z}$. We show that any normalisation of $U_Y(n,d,L)$ is locally factorial and has Picard group isomorphic to $\mathbb{Z}$ (subsection \ref{locfacto}). We deduce that the scheme $R$ of  $(g-1)$-dimensional subspaces of $k^{2g+2}$ which are isotropic for the pencil of quadrics with Segre symbol $[2~ 1 \cdots 1]$ is locally factorial .

              By the (semi)stability of a sheaf on a polarised variety, we mean the slope (semi)stability with respect to the polarisation. In case $X$ is smooth,  the semistability of $\mathcal{U}_{L,x}$ with respect to the polarisation $\theta_L$ and the stability of $\mathcal{U}_L$ with respect to suitable polarisations were proved in \cite{Ba} using Higgs bundles.  In this small note,  we generalise their results to nodal curves. We give a different and incredibly simple proof without using the spectral curves and Higgs bundles. Our main results are the following theorems. 
          
\begin{theorem} \label{thm1} (Theorem \ref{xsmooth})
Let $x\in Y$ be a nonsingular closed point.
\begin{enumerate}
\item The restriction ${\mathcal U}_{L,x}$ of the Poincare sheaf to $x \times U_Y(n,d,L)$ is a vector bundle, it is stable with respect to the polarisation $\theta_L$.
\item The restriction of ${\mathcal U}_{L,x}$ to $x \times U'_Y(n,d,L)$ is stable with respect to any polarisation.
\end{enumerate}
\end{theorem}
               Theorem \ref{thm1} generalises \cite[Proposition 1.4]{Ba} to nodal curves. In fact, it improves \cite[Proposition 1.4]{Ba} even in the case $Y$ is smooth as \cite[Proposition 1.4]{Ba} only proves semistability of ${\mathcal U}_{L,x}$ in the smooth case.

\begin{theorem}  \label{thm2} (Theorem \ref{thmpoincare})
    The Poincar\'e sheaf $\mathcal{U}_{L}$ on $Y \times U_Y(n,d,L)$ is stable with respect to the polarisation  $a \alpha + b \theta_L$  where $\alpha$ is a fixed ample Cartier divisor on $Y$ and $a, b$ are positive integers.  
\end{theorem}
           Theorem \ref{thm2}  generalises \cite[Theorem 1.5]{Ba} to nodal curves. Our proofs of these two theorems were inspired by the ideas in \cite{BBN}. Recently, the results of \cite{BBN} are generalised to nodal curves by myself and my collaborators \cite{ABS}.  We use some computations in this paper in our proofs in \cite{ABS}

\section{Preliminaries}

         In this section, we set up the notation and recall some definitions and results needed.
         
\subsection{Notation} \hfill

      Let $Y$ be an integral projective curve with only $m$ nodes (ordinary double points) as singularities defined over an algebraically closed field. Let $g=h^1(Y, \mathcal{O}_Y)$ be the arithmetic genus of $Y$, we assume that $g \ge 2$. Let 
      $$p : X \longrightarrow Y$$
be the normalisation map. For a node $y_j\in Y$, let $x_j$ and $z_j$ denote the points of $X$ lying over $y_j$. For each $j = 1, \cdots, m$, let $D_j = x_j+z_j$ denote the divisor on $X$.
        
      For a torsion free sheaf ${F}$ on $Y$, let $r(F)$ denote the (generic) rank of $F$ and   
$d({F})=\chi(F)- r({F})\chi(\mathcal{O}_Y)$ denote the degree of ${F}$. Let $\mu(F) = d(F) / r(F)$ denote the slope of $F$. 

                            Let $H$ be an ample line bundle on a variety $Z$ of dimension $m \ge 2$ and $F$ a torsionfree sheaf on $Z$. Then degree of $F$ with respect to $H$, denoted by $d(F)$, is the degree of the restriction of $F$ to a general complete intersection curve on $Z$ of the form $D^{m-1}$ for some divisor $D \in \vert H\vert$.  We remark that if the singular set of $Z$ has codimension at least $2$, then the general complete intersection curve can be chosen to lie in the set of nonsingular points of $Z$. Hence if $Z' \subset Z$ is an open subset contained in the set of nonsingular points of $Z$ such that its complement     
has codimension at least $2$,  the degree of $F$ with respect to $H$ on $Z$ is the same as the degree of $F\vert_{Z'}$ with respect to $H\vert_{Z'}$. 
                 
If $Z$ is projective, then we have 
                 $$d(F) = c_1(F). c_1(H)^{m-1} [Z]\, ,$$
where $[Z]$ denotes the fundamental class of $Z$.
In case Pic $Z \cong \mathbb{Z}$ with $c_1(H)$ the ample generator, one can write $c_1(F)= \lambda c_1(H)$ for some integer $\lambda$ so that $d(F) = \lambda N, N = c_1(H)^m [Z]$. If $Z$ is a  projective space ${\bf P}$, then we take $H$ to be the tautological line bundle $\mathcal{O}_{\bf P}(1)$. Thus $N=1$ for the projective space ${\bf P}$.                   

\subsection{Moduli spaces} \hfill
        
        Let $J(Y) := \ {\rm Pic}^d(Y)$ be the generailsed Jacobian of $Y$ i.e., the moduli space of lines bundles of degree $d$ on $Y$. Let $U_Y(n, d)$ denote the moduli space of stable torsion free sheaves of rank $n$ and degree $d$ on $Y$. It is irreducible \cite{R}, it is a seminormal projective variety of dimension $n^2(g-1)+1$ \cite[Theorem 4.2]{Su}.  Let $U'_Y(n,d)$ be its open dense subvariety corresponding to stable vector bundles on $Y$. From the remark on p.167, section 7 of \cite{N}, one deduces that $U'_Y(n,d)$ is a normal quasi-projective variety, being the GIT quotient of a nonsingular variety $R'$ by PGL$(N)$ for some $N>>0$.  Since $n$ and $d$ are coprime, the points of $U'_Y(n,d)$ correspond to stable vector bundles.  Since the automorphisms of stable bundles are scalars,  $R'$ is a principal PGL$(N)$-bundle over $U'_Y(n,d)$, it follows that $U'_Y(n,d)$ is a nonsingular quasi-projective variety. For a fixed line bundle $L$ on $Y$, let $U'_Y(n,d,L)$ denote the sub variety of $U'_Y(n,d)$ consisting of vector bundles with  determinant isomorphic to $L$. There is a smooth determinant morphism from $U'_Y(n,d)$ onto the nonsingular variety $J(Y)$. It follows that $U'_Y(n,d,L)$ is a nonsingular variety of dimension $(n^2 - 1)(g-1)$.  Let $U_Y(n,d,L)$ denote the closure of $U'_Y(n,d,L)$ in $U_Y(n,d)$ (with reduced structure). 
        
            The moduli varieties $U'_Y(n,d)$ and $U'_Y(n,d,L)$ are locally factorial (\cite[Theorem I]{Bh1}, \cite[Theorem 3A]{B6}). One has 
            $${\rm Pic} \ U'_Y(n,d,L) \cong \mathbb{Z}$$ 
(for $g \ge 2$) \cite[Theorem I]{Bh1}. There is a canonically defined ample line bundle $\theta$ on $U_Y(n,d)$, let 
$$\theta_L   \longrightarrow U_Y(n,d,L)$$ 
be  its restriction. The restriction  of $\theta_L$ to $U'_Y(n,d,L)$ is the generator of Pic $U'_Y(n,d,L)$ \cite[Proposition 3.2]{Bh1}.

\subsection{$(l,m)$-stability for torsionfree sheaves} \hfill

 We make the following definition generalising \cite[Definition 5.1]{NR2}.
\begin{definition}
Let $l$ and $m$ be integers. A torsionfree sheaf $F$ on $Y$ is $(l,m)$-stable if, for every proper subsheaf $G$ of $F$ (with a torsionfree quotient), one has 
$$\frac{d(G) + l}{r(G)} \ < \  \frac{d(F) + l - m}{r(F)}\, .$$   
\end{definition}
        We remark that a torsionfree sheaf $F$ is stable if and only if it is $(0,0)$-stable.

\begin{proposition} \label{L2Bis} \cite[Proposition 2.2]{Bh3}
The $(0,1)$-stable vector bundles $F$, of rank $n$ and determinant $L'$ of degree $d'$ with $n$ coprime to $d'-1$, form an open subset of the moduli space $U_{L' }^s(n,d')$ of stable torsionfree sheaves.
\end{proposition}

\subsection{Local factoriality}  \label{locfacto}    \hfill           

 We do not know if $U_Y(n,d,L)$ is locally factorial or normal.  We make a few observations. 

\begin{lemma} \label{Cl}
$$(1) \ Cl(U_Y(n,d,L)) \cong \mathbb{Z}\, .$$
(2)  Let $\tilde{U}_Y(n,d,L)$ denote a normalisation of $U_Y(n,D,L)$ and $\pi: \tilde{U}_Y(n,d,L) \rightarrow U_Y(n,d,L)$ the normalisation map. Then  $\tilde{U}_Y(n,d,L)$ is locally factorial and 
$${\rm Pic} \ \tilde{U}_Y(n,d,L) \ \cong \ \mathbb{Z}\, .$$      
\end{lemma}
\begin{proof}
(1)  As $U_Y(n,d,L)$ is nonsingular in codimension $1$, by \cite[Proposition 6.5, p.133]{Ha}, we have an isomorphism of class groups
       $$Cl(U'_Y(n,d,L)) \cong Cl(U_Y(n,d,L))\, .$$
Since  $U'_Y(n,d,L)$ is locally factorial, one has 
$Cl(U'_Y(n,d,L)) \cong \ {\rm Pic} \ U'_Y(n,d,L) \cong \mathbb{Z}\, .$ Hence 
$$Cl(U_Y(n,d,L)) \cong \mathbb{Z}\, .$$
(2)  There is a commutative diagram
$$
\begin{array}{ccc}
      {\rm Pic} \ (\tilde{U}_Y(n,d,L) )    & \stackrel{\phi}{\longrightarrow} &  Cl(\tilde{U}_Y(n,d,L))\\
              {}                            &               {}                                &     {}\\
   \downarrow res_P              &             {}                              &          \downarrow {res_C}\\
               {}                            &              {}                              &    {}  \\
     {\rm Pic} \  \pi^{-1}(U'_Y(n,d,L))      &     \stackrel{\phi'}{\longrightarrow}  &   Cl(\pi^{-1}U'_Y(n,d,L))
\end{array}
$$
    As $\pi^{-1}U'_Y(n,d,L) \cong U'_Y(n,d,L), \tilde{U}_Y(n,d,L)$ is nonsingular in codimension $1$. Hence by \cite[Proposition 6.5, p.133]{Ha} the map $res_C$ is an isomorphism. Since  $\pi^{-1}U'_Y(n,d,L)$ is locally factorial, $\phi'$ is an isomorphisms. The map $res_P$ is a surjection as ${\rm Pic} \  \pi^{-1}U'_Y(n,d,L)$ is generated by the restriction of the line bundle $\pi^*\theta_L$ on $\tilde{U}_Y(n,d,L)$. Since $\tilde{U}_Y(n,d,L)$ is normal and $\pi$ being finite, the codimension of the complement of $\pi^{-1}(U'_Y(n,d,L))$ in $\tilde{U}_Y(n,d,L)$ is at least $2$, the map $res_P$ is injective. Then $res_P$ is an isomorphism. From the commutativity of the diagram, it follows that $\phi$ is an isomorphism and hence $\tilde{U}_Y(n,d,L)$ is locally factorial with Picard group isomorphic to the group of integers.  

\end{proof}

\begin{corollary}
Consider a singular pencil of quadrics with Segre symbol $[2~1 \cdots 1]$ given by 
$$q_1 = \sum _{i=1}^{2g} X_i^2 + 2X_0Y_0, \ \
q_2 = \sum _{i=1}^{2g} a_i X_i^2 + (X_0^2 + 2a_0X_0Y_0),$$
with $a_i, i=0,1,\cdots ,2g,$ distinct scalars. 
Let $R$ be the scheme of  $(g-1)$-dimensional subspaces of $k^{2g+2}$ which are isotropic for
this pencil.  Then $R$ is locally factorial. Let $R_0$ be the subscheme of $R$ consisting of those subspaces which contain the unique singular point of the
intersection of quadrics of the pencil. Let $R':= R - R_0$. Then 
$${\rm Pic} \ R = \ {\rm Pic} \ R' = \mathbb{Z}\, .$$ 
\end{corollary}

\begin{proof} Let $X$ be an irreducible reduced projective hyperelliptic  curve of  arithmetic genus $g$ with a single ordinary node as its only singularity, the node is 
 a ramification point of $X$.  To such a curve one can associate a singular pencil of quadrics with Segre symbol $[2~1 \cdots 1]$ \cite{Bh2}. Then $R$ is a normalisation of $U_Y(2,d,L), d$ odd \cite[Theorem 1.2]{Bh2}. More precisely,  there is a morphism $f$ from the scheme $R$ onto $U_Y(2,d,L)$ such that the restrictions of $f$ to $R'$ and $R_0$  are isomorphisms onto $U'_Y(2,d,L)$ and $U_Y(2,d,L) - U'_Y(2,d,L)$.  Hence the corollary follows from Lemma \ref{Cl}.

We note that in particular, if the arithmetic genus $g= 2$, then  the intersection $R$ of quadrics in ${\mathbb P}^5$ is the normalisation of ${U}_Y(2,d,L)$. 

\end{proof}

\begin{remark}
Possibly $Cl(U_Y(n,d,L)) \ncong \ {\rm Pic} \ U_Y(n,d,L)$ as the following example (communicated by V. Srinivas) shows. Let $X$ be the variety obtained by identifying two distinct points $p_1$ and $p_2$ of $\mathbb{P}^2$. Then $Cl(X) \cong Cl(\mathbb{P} - p_1 -p_2) \cong \mathbb{Z}$. However, the Picard group of $X$ is an extension of $\mathbb{Z}$ by $G_m$ as a line bundle on $X$ is obtained identifying the fibres $L_{p_1}$ and $L_{p_2}$ of a line bundle $L$ on $\mathbb{P}^2$.   
\end{remark}

\section{Stability of the bundles $\mathcal{U}_{L,x}$ and $\mathcal{U}_{L}$}

        Since $n$ and $d$ are mutually coprime,  there exits a Poincar\'e sheaf 
         $$\mathcal{U}_L  \longrightarrow Y \times U_Y(n,d,L)\, ,$$ 
such that the restriction $\mathcal{U}_L\vert_{Y\times [E]} \cong E$ for any $E \in U_Y(n,d,L)$ \cite[Theorem 5.12']{N}.  Any two such sheaves differ by the pull back of a line bundle on $U_Y(n,d,L)$.  Since any torsionfree sheaf on $Y$ is locally free on the subset $Y'$ of nonsingular points of $Y$, the restriction of $\mathcal{U}_L$ to $Y' \times U_Y(n,d,L)$ is locally free. It follows that for a nonsingular  point $x$ of $Y$, the restriction $\mathcal{U}_{L,x}$ of the Poincare sheaf to $x \times U_Y(n,d,L)$,  is a vector bundle on $x \times U_Y(n,d,L) \cong U_Y(n,d,L)$. In this section, we study the stability of the restriction $\mathcal{U}_{L,x}$ (for $x$ 
a nonsingular  point of $Y$) with respect to the polarisation $\theta_L$ and the stability of $\mathcal{U}_L$ with respect to suitable polarisations.

\subsection{The morphism $\psi_{F,x}$} \label{morphism}  \hfill

        We fix a nonsingular point $x$ of $Y$ and a $(0,1)$-stable vector bundle $F$ of rank $n$ and determinant $L(x), d(L)=d$ on $Y$.  Let $k_x$ denote the torsion sheaf of length $1$ supported at $x$. Denote by $F_x$ the fibre of $F$ at $x$ and by ${\bf P} := {\bf P}(F_x^*)$ the projective space of lines in $F_x^*$. For every nonzero element $\phi \in {\bf P}(F_x^*)$, we have a nonzero homomorphism $\phi: F \to k_x$  giving an exact sequence 
\begin{equation} \label{hecke1}         
0 \longrightarrow E \longrightarrow F \longrightarrow k_x \longrightarrow 0\, .
\end{equation}

Since $x$ is a nonsingular point of $Y$ and $F$ is locally free, it follows that $E$ is a locally free sheaf of rank $n$ and determinant $L$. Hence  we have the following exact sequence on $Y \times {\bf P}(F_x^*)$ with $\mathcal{E}$ a vector bundle.

\begin{equation} \label{hecke2}         
0 \longrightarrow \mathcal{E} \longrightarrow p_1^*F \longrightarrow \mathcal{O}_{x \times {\bf P}} (1) \longrightarrow 0\, .
\end{equation}

 As in \cite[Lemma 5.5]{NR2}, one can see that $(0,1)$-stability of $F$ implies that $E \in U' _L(n, d)$. 
By the universal property of $U_Y(n,d,L)$, we have a morphism
$$\psi_{F,x} \colon {\bf P}(F_x^*) \longrightarrow U_Y(n,d,L)\, ,$$
such that, for some integer $j$, we have an isomorphism  

\begin{equation} \label{E}
\mathcal{E} \ \cong \ (id \times \psi_{F,x})^*\mathcal{U}_L \otimes p_2^*(\mathcal{O}_{\bf P} ( - j ))\, .
\end{equation}

\begin{lemma} \label{psi-iso}
$\psi_{F,x}$ is an isomorphism onto its image. 
\end{lemma}
\begin{proof}
This can be proved exactly as \cite[Lemma 5.9]{NR2} (\cite[lemma 5.6]{NR2} and \cite[Lemma 3]{Bi} for injectivity). 
We note that $\psi_{F,x}$ maps into $U' _Y(n,d,L)$.
\end{proof}

      Let $\mathcal{E}_x = \mathcal{E}\vert_{x\times \bf{P}}$. There is an exact sequence 
\begin{equation} \label{Ex}
0 \longrightarrow \mathcal{O}_{\bf P} (1)  \longrightarrow \mathcal{E}_x \longrightarrow  \Omega^1_{\bf P}(1) \longrightarrow 0\, ,
\end{equation}
\cite[Lemma 3.1]{BBN}.

\begin{lemma} \label{L3.2} (\cite[Lemma 3.2]{BBN})
   Let $W \subset \mathcal{E}_x$ be a non-zero coherent subsheaf of $\mathcal{E}_x$ such that: \\
   (1) the quotient $\mathcal{E}_x/ W$ is torsionfree, and\\
   (2) $d(W)/r(W) \ge d(\mathcal{E}_x)/r(\mathcal{E}_x)$.\\
 Then $W$ contains the line subbundle $\mathcal{O}_{\bf P} (1)$ of $\mathcal{E}_x$.   
\end{lemma}

\subsection{Stability of $U_{L,x}$}  \hfill

       Let $H$ be an ample line bundle on a variety $Z$ of dimension $m \ge 2$ and $F$ a torsionfree sheaf on $Z$. Then the degree of $F$ with respect to $H$, denoted by $d(F)$, is the degree of the restriction of $F$ to a general complete intersection curve on $Z$ rationally equivalent to  $H^{m -1}$.  We recall that if the singular set of $Z$ has codimension at least $2$, then the general complete intersection curve can be chosen to lie in the set of nonsingular points of $Z$.
       
       If $Z$ is projective, then $d(F) = (c_1(F) . H^{m-1}) [Z]$ where  $[Z]$ denotes the  fundamental class of $Z$.

          Our aim in this section is to prove the following theorem.

\begin{theorem} \label{xsmooth}
Let $x\in Y$ be a nonsingular closed point.
\begin{enumerate}
\item The restriction ${\mathcal U}_{L,x}$ of the Poincare sheaf to $x \times U_Y(n,d,L)$ is stable with respect to the polarisation $\theta_L$.
\item The restriction of ${\mathcal U}_{L,x}$ to $x \times U'_Y(n,d,L)$ is stable with respect to any polarisation.
\end{enumerate}
\end{theorem}

        We recall some constructions and results from \cite{Bh3}.
        
\subsection{}       Each point of  the projective bundle 
       $$P_x \ := {\bf P} (\mathcal{U}_{L,x})$$
corresponds to a pair $(E, \ell)$ where $E \in U_Y(n,d,L)$ and $\ell \in {\bf P}(E_x) \cong {\bf P}({\rm Ext}^1(k_x, E))$ and hence determines an exact sequence of type \eqref{hecke1} and thus a torsionfree sheaf $F$. 
Let $H_x$ be the open subset of $P_x$ defined by 
$$H_x := \{(E,\ell) \in P_x \ \vert \ F\in U_Y(n,d+1,L(x)) \ {\rm is} \ (0,1)-{\rm stable} \}\, .$$         
We have maps
$$
\begin{array}{ccccc}
             H_x& \stackrel{q}{\longrightarrow} & V & \subset & U_Y(n,d+1,L(x))\\
              {}    &               {}         & {} &    {}       &    {}\\
   \downarrow p &             {}        & {} &    {}       &    {}\\
               {}    &               {}         & {} &    {}       &    {}\\
               U_Y(n,d,L)
\end{array}
$$

Here $p$ is the projection $p : H_x \to U_Y(n,d,L)$ defined by $p(E,\ell) = E$.
We have a morphism $q: H_x \rightarrow U_Y(n,d+1,L(x))$ defined by $(E,\ell) \mapsto F$ with image the nonempty open subset $V \subset U_Y(n,d+1,L(x))$ of $(0,1)$-stable vector bundles (Proposition \ref{L2Bis}). 
The fibre of $q$ over $F\in V$ is ${\bf P}(F_x^*)$. The restriction of the projection map $p$ to the fibre ${\bf P}(F_x^*)$  is precisely $\psi_{F,x}$, hence the fibre $P(F_x^*)$ maps isomorphically onto its image $P(F,x) := \psi_{F,x}(P(F_x^*))$  (Lemma \ref{psi-iso}).           
              
\begin{lemma}  \label{l4}
    Let $\mathcal{U}' \subset  \mathcal{U}_{L,x}$ be a subsheaf of rank $r$ with $0 < r < r( \mathcal{U}_{L,x})$. Let $x_1, \cdots, x_p \in Y$ be nonsingular points.
\begin{enumerate}
\item The singular set $S$ of $\mathcal{U}'$ has codimension at least $2$ in $U_Y(n,d,L)$.
\item There is a nonempty open set $U \subset U_Y(n,d,L)$ such that for  $E\in U$, \\
(a) $\mathcal{U}'$ is locally free at $E$,\\
(b) the  homomorphism of fibres $\mathcal{U'}_E \to ( \mathcal{U}_{L,x})_E$ is injective,\\
(c) for all $x_i$ and for the generic line $\ell$ in $E_{x_i}$, the vector bundle $F$ associated to $(E, \ell)$ is $(0,1)$-stable and $\mathcal{U}'$ is locally free at every point of $P(F,x_i)$ outside a subvariety of codimension at least $2$.
\end{enumerate}     
\end{lemma}   
\begin{proof}
   This can be proved as in Lemma \cite[Lemma 4.2]{Bh3} using \cite[Proposition 2.1]{Bh3}. 
\end{proof}

 {\bf Proof of Theorem \ref{xsmooth}}
 \begin{proof}
(1)     Let $\mathcal{U}' \subset  \mathcal{U}_{L,x}$ be a torsionfree subsheaf of rank $r$ with $0 < r < r(\mathcal{U}_{L,x})$ and with a torsionfree quotient. Since $U_Y(n,d,L) -  U'_Y(n,d,L)$ is of codimension at least $3$ by \cite[Proposition 2.1]{Bh3} (or \cite[Theorem 1.3]{BhS} for codimension at least $2$) and $U'_Y(n,d,L)$ is nonsingular, there is an  open subset $Z' \subset U'_Y(n,d,L)$, with $S:= U_Y(n,d,L) - Z'$ of codimension at least $2$ in $U_Y(n,d,L)$, such that $\mathcal{U}'$ is locally free on $Z'$ and $\mathcal{U}'_E \to (\mathcal{U}_{L,x})_E$  is injection for all $E \in Z'$ (Lemma \ref{l4}).  Then $p^{-1} Z' \subset H_x$ is a Zariski open subset with codim $H_x - p^{-1}Z' \ge 2$ so that dim $S \le $ dim $U_Y(n,d,L) + n - 3 = $ dim $U_Y(n,d+1,L(x)) + n - 3$. The image $V$ of $q$ has dimension equal to the dimension of $U_Y(n,d+1,L(x))$, hence the general fibre of $q$ intersects $S$ in a closed subset of dimension at most $n-3$. Therefore for a general $(0,1)$-stable vector bundle $F \in U'_{L(x)}(n,d+1)$, the complement of $\psi_{F,x}^{-1}(Z')$ has codimension at least $2$. By generality of $F$, we can assume that $F$ is defined by a pair $(E,\ell)$ and $\ell$ is not in the fibre $P(\mathcal{U}')_E$ i.e. the line determined by $\ell$ is not contained in the fibre of $\mathcal{U}'$ at $E$ (as in \cite[Proof of theorem 3.6]{BBN}).  
     
    For a sheaf $N$ on $\bf{P}$, define $N(-j):= N\otimes {\mathcal O}_{\bf P}(- j )$.  By equation \eqref{E}, one has 
     $$\psi_{F,x}^* (\mathcal{U}_{L,x}) (- j ) \cong \mathcal{E}_x\, ,$$
  where  $\mathcal{E}_x$ is as in \eqref{Ex}. 
   Hence 
  $$\psi_{F,x}^*(\mathcal{U}')(- j) \vert_{\psi_{F,x}^{-1}Z'} \subset \mathcal{E}'_x = \mathcal{E}_x \vert _{\psi_{F,x}^{-1}Z'}\, .$$ 
  The condition that the line determined by $\ell$ is not contained in the fibre of $\mathcal{U}'$ at $E$  implies that $\psi_{F,x}^*(\mathcal{U}')(- j) \vert_ {\psi_{F,x}^{-1}Z'}$ does not contain the line subbundle $\mathcal{O}_P(1)\vert_{\psi_{F,x}^{-1}Z'}$ of $\mathcal{E}'_x$. For a general $F$, the complement of $\psi_{F,x}^{-1}(Z')$ in $\bf{P}$ has codimension at least $2$, hence  $\psi_{F,x}^*(\mathcal{U}')( - j)\vert_{\psi_{F,x}^{-1}Z'}$ can be extended to $\bf{P}$. By Lemma \ref{L3.2} (applied to an extension of $\psi_{F,x}^*(\mathcal{U}')( - j)\vert_{\psi_{F,x}^{-1}Z'}$ to $\bf{P}$), we have 
  $$\mu(\psi_{F,x}^*(\mathcal{U}')( - j)) < \mu(\mathcal{E}_x) = \mu(\psi_{F,x}^*(\mathcal{U}_{L,x})(- j ))$$ 
  so that 
  \begin{equation}\label{mu1}
  \mu(\psi_{F,x}^*(\mathcal{U}')) <  \mu(\psi_{F,x}^*(\mathcal{U}_{L,x}))\, .
  \end{equation}
  One has $c_1( \mathcal{U}'\vert_{Z'}) = c_1(\mathcal{U}'\vert_{U'_Y(n,d,L)}) = \lambda_{\mathcal{U}'} c_1(\theta_L\vert_{U'_Y(n,d,L)})$ for some scalar $\lambda_{\mathcal{U}'}$. Since $d(\mathcal{U}') = d(\mathcal{U}'\vert_{U'_Y(n,d,L)})$, this implies that $d(\mathcal{U}') = \lambda_{\mathcal{U}'} d(\theta_L)$, where $d(\theta_L) = \theta_L^{(n^2-1)(g-1)} [U_Y(n,d,L)].$   Similarly, $d(\mathcal{U}_{L,x}) = \lambda_{\mathcal{U}_{L,x}} d(\theta_L)$. We have 
  $d(\psi_{F,x}^*(\mathcal{U}')) = \lambda_{\mathcal{U}'} d(\psi_{F,x}^*\theta_L).$ 
  Since $\theta_L$ is an ample line bundle, by Lemma \ref{psi-iso}, $\psi_{F,x}^*(\theta_L)$ is an ample line bundle on $\bf{P}$. As Pic ${\bf P}= \mathbb{Z}$ is generated by $\mathcal{O}_{\bf{P}}(1)$, we have $\psi_{F,x}^*(\theta_L) = \mathcal{O}_{\bf{P}}(\delta)$ for some $\delta > 0$.  Hence one gets $d(\psi_{F,x}^*(\mathcal{U}')) =  \lambda_{\mathcal{U}'} \delta$. Similarly, 
   $d(\psi_{F,x}^*(\mathcal{U}_{L,x})) =  \lambda_{\mathcal{U}_{L,x}} \delta$. Thus the equation \eqref{mu1} is equivalent to 
$$\frac{\lambda_{\mathcal{U}'}}{r(\mathcal{U}')} < \frac{\lambda_{\mathcal{U}_{L,x}}}{r(\mathcal{U}_{L,x})}\, .$$
 Hence 
$$\frac{\lambda_{\mathcal{U}'}d(\theta_L)}{r(\mathcal{U}')} < \frac{\lambda_{\mathcal{U}_{L,x}}d(\theta_L)}{r(\mathcal{U}_{L,x})}\, .$$
 Thus $\mu(\mathcal{U}') < \mu(\mathcal{U}_{L,x})$, proving the stability of $\mathcal{U}_{L,x}$.\\
(2)    Note that $\psi_{F,x} ({\bf P}) \subset U'_Y(n,d,L)$. Hence Part (2) can be proved similarly as Part (1) using the facts that the singular set $U_Y(n,d,L) - U'_Y(n,d,L)$ is of codimension at least $2$ in $U_Y(n,d,L)$ and Pic $U'_Y(n,d,L) \cong \mathbb{Z}$, with generator $\theta_L$.
     
\end{proof}

\subsection{Stability of the Poincar\'e bundle on $Y \times U_Y(n,d,L)$} \hfill

         The moduli space $U_Y(n,d,L)$ is unirational \cite[Lemma 3.5(1)]{BhS}. Hence there is no non-constant map from $U_Y(n,d,L)$ to Pic $Y$. Using the see-saw theorem \cite[Corollary 6, p. 54]{Mu}, one sees that every line bundle on $Y \times U_Y(n,d,L)$ is the tensor product of the pull back of a line bundle on $U_Y(n,d,L)$ and the pull back of a line bundle on $Y$. Hence we take the polarisation on $Y\times U_Y(n,d,L)$ represented by a divisor of the form $a \alpha + b \theta_L$, where $\alpha$ is a fixed ample Cartier divisor on $Y$ and $a, b$ are positive integers.  
         
\begin{theorem} \label{thmpoincare}
    The Poincar\'e bundle $\mathcal{U}_{L}$ on $Y \times U_Y(n,d,L)$ is stable with respect to the polarisation  $a \alpha + b \theta_L$ with $a, b$ positive integers.
\end{theorem}
\begin{proof}
    Let $Q$ be a torsionfree quotient sheaf of $\mathcal{U}_{L}$. It suffices to check that 
    $$ r(\mathcal{U}_{L}) d(Q) - r(Q) d(\mathcal{U}_{L}) \ > \ 0\, .$$ 
    
    Let $E \in U'_Y(n,d,L)$ be a general element. Let $Q_Y = Q\vert_{Y\times E}$. Then $Q_Y$ is a quotient of the locally free sheaf $E$. The quotient $Q_Y$ may not be torsionfree, let $Q'_Y := Q_Y/tor$ be the quotient of $Q_Y$ by its torsion subsheaf. Let $Y'$ be the set of smooth points of $Y$. 
    
    Since $Y' \times U'_Y(n,d,L)$ is smooth, there is an open subset $Z' \subset Y' \times U'_Y(n,d,L)$ whose complement in $Y' \times U'_Y(n,d,L)$ is of codimension at least $2$ such that $Q$ is locally free over $Z'$. For a general $E \in U'_Y(n,d,L)$, $Z'$ contains  the curve $Y' \times E$. This implies that $Q_Y$ has torsion at most at nodes of $Y$ and $r(Q)= r(Q_Y)= r(Q'), d(Q_Y) \ge d(Q'_Y)$ on $Y$. Since $E$ is stable,  one has $r(E) d(Q'_Y) - r(Q'_Y) d(E) > 0$. This implies that 
\begin{equation}\label{Y1}    
    r( \mathcal{U}_{L}) d(Q_Y) - r(Q) d(E) > 0\, . 
\end{equation}

         Since the complement of $Z'$ in $Y' \times U'_Y(n,d,L)$ is of codimension at least $2$, for a general $x \in  Y'$, $Z' \cap (x \times U'_Y(n,d,L))$ has complement of  codimension at least $2$ $x \times U'_Y(n,d,L)$. Since $U_Y(n,d,L) - U'_Y(n,d,L)$ has codimension at least $3$ (\cite[Proposition 2.1]{Bh3}), it follows that the open subset $Z' \cap (x \times U'_Y(n,d,L))$ has complement of  codimension at least $2$ in $x \times U_Y(n,d,L)$. Since $x$ is a non-singular point, the sheaf $Q_x := Q\vert_{x \times U_Y(n,d,L)}$ is locally free on $Z' \cap (x \times U_Y(n,d,L))$ and $d(Q_x) = d(Q_x\vert_{Z' \cap (x \times U'_Y(n,d,L))})$.  For $x$ a non-singular point, $\mathcal{U}_{L,x} =  \mathcal{U}\vert_{x \times U_Y(n,d,L)}$ is a stable vector bundle by Theorem \ref{xsmooth}.  Hence,
$r(\mathcal{U}_{L,x}) d(Q_x) - r(Q_x) d(\mathcal{U}_{L,x}) >0$, i.e.
\begin{equation}\label{x1}    
    r(\mathcal{U}_{L}) d(Q_x) - r(Q) d(\mathcal{U}_{L,x}) > 0\, . 
\end{equation}
Let $m:= \ {\rm dim} \ U_Y(n,d,L)$.
One has 
$$
\begin{array}{ccl}
d(\mathcal{U}_L) & = &  (c_1(\mathcal{U}_L).(b \theta_L + a \alpha)^m) [Y \times U_Y(n,d,L)]\\
                   {} & = & c_1(\mathcal{U}_L). (m a b^{m-1} \alpha \theta_L^{m-1} + b^m \theta_L^m) [Y \times U_Y(n,d,L)]\\
                         {} & = & m a b^{m-1} \alpha [Y] d(\mathcal{U}_{L,x}) + b^m d(E) \theta_L^m [U_Y(n,d,L)] \\         
\end{array}
$$
Similarly, 
$$d(Q) = m a b^{m-1} \alpha [Y] d(\mathcal{Q}_x) + b^m d(Q_Y) \theta_L^m [U_Y(n,d,L)] \, .$$
Then
$$
\begin{array}{ccl}
r(\mathcal{U}_L) d(Q) - r(Q) d(\mathcal{U}_L) & = & m a b^{m-1} \alpha [Y] (r(\mathcal{U}_L) d(\mathcal{Q}_x) 
- r(Q) d(\mathcal{U}_{L,x}))\\
 {} & {} & + b^m  \theta_L^m [U_Y(n,d,L)] ( r(\mathcal{U}_L) d(Q_Y) - r(Q) d(E) )\\
 {} & > & 0\, , 
\end{array}
$$
by \eqref{Y1} and \eqref{x1}.

This completes the proof of the theorem.
\end{proof}

\begin{corollary}
The Poincar\'e bundle $\mathcal{U}'_{L}$ on $Y \times U'_Y(n,d,L)$ is stable with respect to any polarisation.  
\end{corollary}

\begin{proof}
      This follows as in Theorem \ref{thmpoincare}  using Theorem \ref{xsmooth}(2) and the facts that the Picard group of $U'_Y(n,d,L)$ is isomorphic to $\mathbb{Z}$ and it is generated by $\theta_L$ \cite[Theorem I]{Bh1}. 
\end{proof}

\end{document}